\documentclass[11pt]{article}

\usepackage[T1]{fontenc}
\usepackage{lmodern}
\usepackage{microtype}
\usepackage[a4paper,margin=25mm]{geometry}
\usepackage{amsmath,amssymb,amsthm,mathtools}

\newtheorem{theorem}{Theorem}[section]
\newtheorem{proposition}[theorem]{Proposition}
\newtheorem{lemma}[theorem]{Lemma}
\newtheorem{corollary}[theorem]{Corollary}

\theoremstyle{remark}
\newtheorem{remark}[theorem]{Remark}

\newcommand{\spen}{s^{+}}
\newcommand{\sneg}{s^{-}}
\newcommand{\tr}{\operatorname{tr}}
\newcommand{\Fnorm}[1]{\left\lVert #1\right\rVert_{\mathrm F}}

\title{Positive and Negative Square Energies of $2$-Connected Graphs}
\author{S. Akbari$^a$, Fu-Tao Hu$^b$\thanks{E-mail address: hufu@ahu.edu.cn}, Ya-Yang Liu$^b$ \\
{\small $^a$Department of Mathematics, Sharif University of Technology, Tehran, Iran}\\
{\small $^b$Center for Pure Mathematics, School of Mathematical Sciences, Anhui University, Hefei, P.R. China}
}
\date{}

\begin{document}

\maketitle

\begin{abstract}
Let $G$ be a graph of order $n$, and let $s^+(G)$ and $s^-(G)$ denote the sums of the squares of the positive and negative adjacency eigenvalues of $G$, respectively. Recently, Liu, Tang, and Zhang proved the conjecture of Elphick, Farber, Goldberg, and Wocjan that every connected graph $G$ of order $n$ satisfies
$
\min\{s^+(G), s^-(G)\} \ge n-1.
$
For positive square energy, we strengthen this result by showing that
every $2$-connected graph $G$ of order $n$ which is not a cycle
satisfies $s^+(G)\ge n$.  The formally analogous assertion for
$s^-$ is false: the complete graph $K_n$ satisfies
$s^-(K_n)=n-1$.  We prove a natural counterpart in the triangle-free
class: every triangle-free $2$-connected noncycle $G$
satisfies
$
 \min\{s^+(G),s^-(G)\}>n.
$
More generally, it is enough that some maximum-degree vertex of $G$
belongs to no triangle.
Together with the exact square energies of cycles, this characterizes
the triangle-free $2$-connected graphs for which $s^-(G)\ge n$;
the only exceptions are the cycles $C_{4k+3}$ with $k\geq1$.
\end{abstract}



\section{Introduction}

All graphs in this paper are finite, simple, and undirected.  Let
$G$ be a graph of order $n$ and size $m$.  If
$
 \lambda_1(G)\geq\cdots\geq\lambda_n(G)
$
are the adjacency eigenvalues of $G$, define
$$
 \spen(G)=\sum_{\lambda_i(G)>0}\lambda_i^2(G),
 \qquad
 \sneg(G)=\sum_{\lambda_i(G)<0}\lambda_i^2(G).
$$
Since
$
 \spen(G)+\sneg(G)=\tr(A^2(G))=2m,
$
these quantities split the second adjacency spectral moment according
to the signs of the eigenvalues. Elphick, Farber, Goldberg, and Wocjan \cite{EFGW} conjectured that
every connected graph $G$ of order $n$ satisfies
$
 \min\{\spen(G),\sneg(G)\}\geq n-1.
$
This conjecture was recently proved by Liu, Tang, and Zhang
\cite{LTZ}.  Before that resolution, Akbari, Kumar, Mohar, Pragada,
and Zhang \cite{AKMPSZ} proposed the following strengthening for
positive square energy:
\begin{equation}\label{eq:akbari-conjecture}
 \text{if $G$ is connected and $m\geq n+1$, then }\spen(G)\geq n.
\end{equation}
They verified \eqref{eq:akbari-conjecture} for several graph classes,
including claw-free graphs and graphs of diameter two.
We prove \eqref{eq:akbari-conjecture} for $2$-connected graphs.
More precisely, every $2$-connected noncycle of order $n$ satisfies
$\spen(G)\geq n$.  This also implies the same conclusion for every
Hamiltonian noncycle.

It is essential that this statement concerns positive square energy.
There is no unrestricted negative counterpart: for every $r\geq 2$,
the eigenvalues of $K_r$ are $r-1$ once and $-1$ with multiplicity
$r-1$.  Hence
\[
 \sneg(K_r)=r-1<r.
\]
Thus the Liu--Tang--Zhang lower bound $n-1$ is already best possible
for negative square energy even among $2$-connected non-cycle graphs.
Our negative result identifies a natural class in which the stronger
bound does hold.  We prove that every triangle-free $2$-connected
non-cycle graph satisfies $\sneg(G)>n$.  Together with the cycle formulas,
this gives an exact classification within the triangle-free
$2$-connected class.

\section{Preliminaries}\label{sec:preliminaries}

We begin with the spectral decomposition of the adjacency matrix,
the Liu--Tang--Zhang Theorem, and the variational descriptions of
the positive and negative spectral parts.  Let $G$ be a graph of
order $n$, with adjacency matrix $A=A(G)$ and eigenvalues
$\lambda_1,\ldots,\lambda_n$, counted with multiplicity.  Define
$$
 \spen(G)=\sum_{\lambda_i>0}\lambda_i^2,
 \qquad
 \sneg(G)=\sum_{\lambda_i<0}\lambda_i^2.
$$
Write
$$
 A=P-N,
 \qquad P=A_+\succeq0,
 \qquad N=A_-\succeq0,
 \qquad PN=0,
$$
where $P$ and $N$ are the positive and negative spectral parts of $A$.
Then
$$
 \spen(G)=\tr(P^2)=\Fnorm{P}^{2},
 \qquad
 \sneg(G)=\tr(N^2)=\Fnorm{N}^{2}.
$$

The following interesting result was recently  proved in \cite{LTZ}.

\begin{theorem}[Liu--Tang--Zhang] 
\label{thm:connected}
If $F$ is a connected graph, then
$$
 \min\{\spen(F),\sneg(F)\}\geq |V(F)|-1.
$$
In particular, both square energies are at least $|V(F)|-1$.
\end{theorem}

We shall use the following variational description; the second
identity follows from the first one by replacing $B$ with $-B$.

\begin{lemma}\label{lem:variational}
For every real symmetric matrix $B$,
\[
 \Fnorm{B_+}^{2}
 =\min_{M\succeq0}\Fnorm{B+M}^{2},
 \qquad
 \Fnorm{B_-}^{2}
 =\min_{M\succeq0}\Fnorm{B-M}^{2}.
\]
\end{lemma}

The positive identity in Lemma~\ref{lem:variational} appears in
\cite{ZhangExtremal}.  We shall also use the following superadditivity
result of Akbari, Kumar, Mohar, and Pragada \cite{AkbariLinear}.

\begin{lemma}
\label{lem:superadd}
If $H_1,\ldots,H_k$ are pairwise vertex-disjoint induced subgraphs of
a graph $G$, then
$$
 \spen(G)\geq\sum_{i=1}^k\spen(H_i),
 \qquad
 \sneg(G)\geq\sum_{i=1}^k\sneg(H_i).
$$
\end{lemma}

The negative inequality follows either from the same proof or by
applying its matrix form to $-A(G)$.  For a symmetric matrix $Q$
indexed by $V(G)$ and $v\in V(G)$, define
\[
 \mu_v(Q)=Q_{vv}^2+2\sum_{u\neq v}Q_{vu}^2.
\]
This is precisely the part of $\Fnorm{Q}^2$ supported on the row and
column indexed by $v$.

\begin{lemma}\label{lem:deletion}
For every $v\in V(G)$,
\[
 \spen(G)\geq\spen(G-v)+\mu_v(P),\quad \sneg(G)\geq\sneg(G-v)+\mu_v(N).
\]
\end{lemma}

\begin{proof}
Put $U=V(G)\setminus\{v\}$.  Since
$$
 A(G-v)+N[U]=P[U]
 \qquad\text{and}\qquad N[U]\succeq0,
$$
Lemma~\ref{lem:variational} gives
$$
 \spen(G-v)\leq\Fnorm{P[U]}^{2}.
$$
Moreover,
$$
 \Fnorm{P}^{2}
 =\Fnorm{P[U]}^{2}+P_{vv}^2+2\sum_{u\neq v}P_{vu}^2.
$$
Combining the last two displays proves the result.

For the negative part, observe that
\[
 A(G-v)-P[U]=-N[U],
 \qquad P[U]\succeq0.
\]
The second identity in Lemma~\ref{lem:variational} therefore gives
\[
 \sneg(G-v)\leq\Fnorm{N[U]}^2.
\]
Finally,
\[
 \Fnorm{N}^2
 =\Fnorm{N[U]}^2+N_{vv}^2+2\sum_{u\neq v}N_{vu}^2,
\]
which proves the negative inequality.
\end{proof}

\section{Main Results}

We first prove the positive result for all $2$-connected graphs.
The triangle-free local estimate is invariant under reflection of the
spectrum about zero, and hence yields strict lower bounds for both
square energies.  We then combine these bounds with the exact cycle
formulas.

\begin{lemma}\label{lem:dense}
Suppose that $G$ contains a subgraph $H$ with vertex set $S$, where
$|S|=r$ and $|E(H)|=h$.  Then there exists $v\in S$ such that
$$
 \mu_v(P)\geq \frac{8h^2}{r^2(r+1)},
$$
where $H$ need not be induced.
\end{lemma}

\begin{proof}
Let $X=P[S]$ and let $J=G[S]$.  Since $P-A=N\succeq0$, taking the
principal submatrix indexed by $S$ gives
$$
 X-A(J)=N[S]\succeq0.
$$
For the all-ones vector $\mathbf1\in\mathbb R^r$, it follows that
\begin{equation}\label{eq:sum-lower}
 \mathbf1^{\mathsf T}X\mathbf1
 \geq \mathbf1^{\mathsf T}A(J)\mathbf1
 =2|E(J)|
 \geq2h.
\end{equation}

For $v\in S$, put
$$
 \tau_v=P_{vv}^{2}
       +2\sum_{u\in S\setminus\{v\}}P_{vu}^{2}.
$$
Consider the following $r+\binom r2=r(r+1)/2$ real numbers:
$$
 P_{vv}\quad(v\in S),
 \qquad
 2P_{uv}\quad(\{u,v\}\in\tbinom S2).
$$
Their sum is $\mathbf1^{\mathsf T}X\mathbf1$, while their squared sum is
$$
 \sum_{v\in S}P_{vv}^{2}
 +4\sum_{\{u,v\}\in\binom S2}P_{uv}^{2}
 =\sum_{v\in S}\tau_v.
$$
Therefore, Cauchy--Schwarz and \eqref{eq:sum-lower} imply
$$
 4h^2
 \leq\bigl(\mathbf1^{\mathsf T}X\mathbf1\bigr)^2
 \leq\frac{r(r+1)}{2}\sum_{v\in S}\tau_v.
$$
Hence
$
 \sum_{v\in S}\tau_v\geq\frac{8h^2}{r(r+1)}
$
and so there exists $v\in S$ such that
$
 \tau_v\geq\frac{8h^2}{r^2(r+1)}.
$
Finally,
$$
 \mu_v(P)
 =\tau_v+2\sum_{u\notin S}P_{vu}^{2}
 \geq\tau_v,
$$
which proves the lemma.
\end{proof}

\begin{corollary}\label{cor:triangle}
If $G$ contains a triangle $T$, then there is a vertex $v\in V(T)$
such that
$
 \mu_v(P)\geq2.
$
\end{corollary}

The following lemma is the main estimate.  It is written entirely in
terms of finite weighted sums of adjacency eigenvalues.

\begin{lemma}\label{lem:local}
Let $G$ be a graph, let $d=\Delta(G)\geq3$, and let $d_G(v)=d$.
If $v$ belongs to no triangle, then
\[
 \mu_v(P)>2
 \qquad\text{and}\qquad
 \mu_v(N)>2.
\]
\end{lemma}

\begin{proof}
We prove the assertion for $P$.  The entire argument uses only the
four local spectral moments in \eqref{eq:weighted-identities} below
and the bound $|\lambda_i|\leq d$.  These data are invariant under
$A\mapsto-A$, while that replacement interchanges $P$ and $N$.
Consequently, the same proof will also establish the assertion for
$N$.

Choose an orthonormal eigenbasis $z_1,\ldots,z_n$ of $A$, where
$Az_i=\lambda_i z_i$, and define
$w_i=z_i^2(v)$ for $i=1,\ldots,n$.
Thus $w_i\geq0$.  By the spectral decomposition we find that:
$$
 A^k=\sum_{i=1}^n\lambda_i^k z_i z_i^{\mathsf T}.
$$
We have
\begin{equation}\label{eq:weighted-identities}
 \sum_{i=1}^n w_i=1,
 \qquad
 \sum_{i=1}^n w_i\lambda_i=0,
 \qquad
 \sum_{i=1}^n w_i\lambda_i^2=d,
 \qquad
 \sum_{i=1}^n w_i\lambda_i^3=0.
\end{equation}
Indeed, the four sums in \eqref{eq:weighted-identities} are respectively
$1,A_{vv},(A^2)_{vv},(A^3)_{vv}$.  Here $A_{vv}=0$,
$(A^2)_{vv}=d_G(v)=d$, and $(A^3)_{vv}$ is twice the number of
triangles containing $v$, so it is zero.  Also,
$|\lambda_i|\leq\Delta(G)=d$ for every $i$.
Let
$$
 I_+=\{i:\lambda_i>0\},
 \qquad
 I_-=\{i:\lambda_i<0\}.
$$
Indices corresponding to zero eigenvalues belong to neither set, but
their weights remain present in $\sum_{i=1}^n w_i=1$.
Define the following  sums:
\begin{align*}
 a&=\sum_{i\in I_+}w_i\lambda_i
     =\sum_{i\in I_-}w_i(-\lambda_i),\\
 c&=\sum_{i\in I_+}w_i\lambda_i^2,\\
 e&=\sum_{i\in I_-}w_i\lambda_i^2=d-c,\\
 h&=\sum_{i\in I_+}w_i\lambda_i^3
     =\sum_{i\in I_-}w_i(-\lambda_i)^3.
\end{align*}
 Since
$
 P=\sum_{i\in I_+}\lambda_i z_i z_i^{\mathsf T}$ and 
$ N=\sum_{i\in I_-}(-\lambda_i)z_i z_i^{\mathsf T},
$
we have $P_{vv}=a$, $(P^2)_{vv}=c$, and $(N^2)_{vv}=e$.  Moreover, since $P$
is symmetric,
$$
 \sum_{u\in V(G)}P_{vu}^2=(P^2)_{vv}=c.
$$
Consequently,
\begin{equation}\label{eq:mass-c}
 \mu_v(P)=2(P^2)_{vv}-P_{vv}^2=2c-a^2.
\end{equation}

Let
$
 \pi_+=\sum_{i\in I_+}w_i$
 and
 $\pi_-=\sum_{i\in I_-}w_i.
$
Both $c$ and $e$ are positive.  Indeed, if either one were zero, the
first-power identity in \eqref{eq:weighted-identities} would force the
other one to be zero as well, contradicting
$\sum_iw_i\lambda_i^2=d>0$.  Applying Cauchy--Schwarz separately to
the two  sums defining $a$ gives
$$
 \begin{aligned}
 a^2
 &=\left(\sum_{i\in I_+}w_i\lambda_i\right)^2
 \leq
 \left(\sum_{i\in I_+}w_i\right)
 \left(\sum_{i\in I_+}w_i\lambda_i^2\right)
 =\pi_+c,\\
 a^2
 &=\left(\sum_{i\in I_-}w_i(-\lambda_i)\right)^2
 \leq
 \left(\sum_{i\in I_-}w_i\right)
 \left(\sum_{i\in I_-}w_i\lambda_i^2\right)
 =\pi_-e.
 \end{aligned}
$$
Consequently,
$$
 a^2\left(\frac1c+\frac1e\right)
 \leq\pi_++\pi_-\leq1,
$$
and hence
\begin{equation}\label{eq:a-bound}
 a^2\leq\frac{ce}{c+e}=\frac{ce}{d}.
\end{equation}
By \eqref{eq:mass-c} and \eqref{eq:a-bound},
\begin{equation}\label{eq:basic-mass}
 \mu_v(P)
 \geq2c-\frac{c(d-c)}d
 =c+\frac{c^2}{d}.
\end{equation}

We first suppose that $d\geq4$.  A further application of
Cauchy--Schwarz to the negative-eigenvalue sum gives
$$
 \begin{aligned}
 e^2
 &=\left(\sum_{i\in I_-}w_i(-\lambda_i)^2\right)^2\\
 &\leq
 \left(\sum_{i\in I_-}w_i(-\lambda_i)\right)
 \left(\sum_{i\in I_-}w_i(-\lambda_i)^3\right)
 =ah.
 \end{aligned}
$$
For every $i\in I_+$, $0<\lambda_i\leq d$, and hence
$$
 h=\sum_{i\in I_+}w_i\lambda_i^3
 \leq d\sum_{i\in I_+}w_i\lambda_i^2=dc.
$$
Using \eqref{eq:a-bound}, we obtain
$
 e^2
 \leq adc
 \leq dc\sqrt{\frac{ce}{d}}
 =\sqrt d\,c^{3/2}e^{1/2}.
$
It follows that
$
 e\leq d^{1/3}c$ and $
 c\geq\frac{d}{1+d^{1/3}}.
$
Because $c+c^2/d$ is increasing for $c\geq0$, inequality
\eqref{eq:basic-mass} yields
\begin{equation}\label{eq:dge4}
 \mu_v(P)
 \geq
 \frac{d(d^{1/3}+2)}{(d^{1/3}+1)^2}.
\end{equation}
Writing $x=d^{1/3}$, the right-hand side of \eqref{eq:dge4} becomes
$
 F(x)=\frac{x^3(x+2)}{(x+1)^2}.
$
Its derivative is
$$
 F'(x)=\frac{2x^2(x^2+3x+3)}{(x+1)^3}>0.
$$
At $d=4$, we have $x^2=\sqrt[3]{16}<3$, and therefore
$
 F(x)=\frac{4(x+2)}{(x+1)^2}>2.
$
Thus \eqref{eq:dge4} is strictly greater than $2$ for every $d\geq4$.
It remains to consider $d=3$.  For a real number $t$, write
$t_+=\max\{t,0\}$.  Define the cubic polynomial
$$
 \psi(t)=
 \frac{121}{1029}
 \left(t+\frac94\right)^2
 \left(t-\frac{27}{121}\right).
$$
For $-3\leq t\leq0$, we have $\psi(t)\leq0=(t_+)^2$.  For
$0\leq t\leq3$, the exact factorization
$$
 t^2-\psi(t)
 =\frac{121}{1029}
  \left(t-\frac{27}{44}\right)^2(3-t)
 \geq0
$$
shows that $\psi(t)\leq(t_+)^2$.  Hence, because
$|\lambda_i|\leq3$,
$
 c=\sum_{i=1}^n w_i(\lambda_i)_+^2
 \geq\sum_{i=1}^n w_i\psi(\lambda_i).
$
Expanding $\psi$ gives
$$
 \psi(t)
 =-\frac{729}{5488}
  +\frac{2619}{5488}t
  +\frac{345}{686}t^2
  +\frac{121}{1029}t^3.
$$
Substituting this expansion and the four finite-sum identities in
\eqref{eq:weighted-identities}, with $d=3$, we conclude that
$$
 \begin{aligned}
 c
 &\geq\sum_{i=1}^n w_i\psi(\lambda_i)\\
 &=-\frac{729}{5488}\sum_{i=1}^n w_i
   +\frac{2619}{5488}\sum_{i=1}^n w_i\lambda_i
   +\frac{345}{686}\sum_{i=1}^n w_i\lambda_i^2
   +\frac{121}{1029}\sum_{i=1}^n w_i\lambda_i^3\\
 &=-\frac{729}{5488}+3\left(\frac{345}{686}\right)
 =\frac{7551}{5488}.
 \end{aligned}
$$
Finally, \eqref{eq:basic-mass} gives
$$
 \mu_v(P)
 \geq
 \frac{7551}{5488}
 +\frac13\left(\frac{7551}{5488}\right)^2
 =\frac{60445755}{30118144}
 =2+\frac{209467}{30118144}
 >2.
$$
This proves the assertion for $P$.  Replacing every $\lambda_i$ by
$-\lambda_i$ leaves \eqref{eq:weighted-identities} and the spectral
bound unchanged and interchanges the two spectral parts.  It therefore
gives the identical estimate for $\mu_v(N)$.
\end{proof}

\begin{theorem}\label{thm:main}
Let $G$ be a $2$-connected graph of order $n$.  If
$\Delta(G)\geq3$, then $\spen(G)\geq n$.
If, in addition, some maximum-degree vertex of $G$ belongs to no
triangle, then
\[
 \min\{\spen(G),\sneg(G)\}>n.
\]
In particular, the strict inequality holds, when $G$ is triangle-free.
\end{theorem}

\begin{proof}
Suppose first that $G$ contains a triangle.  By
Corollary~\ref{cor:triangle}, there exists a vertex $v$ such that
$\mu_v(P)\geq2$.
Since $G$ is $2$-connected, $G-v$ is connected.
Theorem~\ref{thm:connected} and Lemma~\ref{lem:deletion} give
$$
 \spen(G)
 \geq\spen(G-v)+\mu_v(P)
 \geq(n-2)+2=n.
$$

It remains to handle the triangle-free case of the first assertion,
this will follow from the stronger conclusion.  Thus let $v$ be a
maximum-degree vertex which
belongs to no triangle.  Lemma~\ref{lem:local} gives
$\mu_v(P)>2$ and $\mu_v(N)>2$.  Again, $G-v$ is connected, so
\[
 \spen(G)
 \geq\spen(G-v)+\mu_v(P)
 >(n-2)+2=n.
\]
The negative deletion inequality gives, in exactly the same way,
\[
 \sneg(G)
 \geq\sneg(G-v)+\mu_v(N)
 >(n-2)+2=n.
\]
If $G$ is triangle-free, every maximum-degree vertex satisfies the
additional hypothesis, so the proof is complete.
\end{proof}

\noindent The following result was proved in \cite{AKMPSZ}.

\begin{theorem}\label{prop:cycles}

For $n\geq3$, we have, 
$$
 \begin{array}{c|c|c}
  \text{ }n & \spen(C_n) & \sneg(C_n)\\ \hline
  n\text{ even}
   & n & n\\[2mm]
  n\equiv3\pmod4
   & n-1+\sec(\pi/n) & n+1-\sec(\pi/n)\\[2mm]
  n\equiv1\pmod4
   & n+1-\sec(\pi/n) & n-1+\sec(\pi/n)
 \end{array}
$$
In particular,
$$
 \spen(C_n)\geq n
 \text{ if and only if }
 n\not\equiv1\pmod4,
 \qquad
 \sneg(C_n)\geq n
 \text{ if and only if }
 n\not\equiv3\pmod4.
$$
\end{theorem}

\vspace{0.3cm}
\begin{theorem}
\label{thm:classification}
Let $G$ be a $2$-connected graph of order $n$.  Then
$
 \spen(G)\geq n
 \text{ if and only if }
 G\not\cong C_{4k+1}\ \text{for every }k\geq1.
$
\end{theorem}

\begin{proof}
Every $2$-connected graph has minimum degree at least two.  If
$\Delta(G)=2$, then $G$ is a cycle, and
Theorem~\ref{prop:cycles} gives the result.  If
$\Delta(G)\geq3$, Theorem~\ref{thm:main} gives
$\spen(G)\geq n$.  Conversely,
Theorem~\ref{prop:cycles} gives
$$
 \spen(C_{4k+1})
 =4k+2-\sec\left(\frac{\pi}{4k+1}\right)
 <4k+1.
$$
\end{proof}

\begin{proposition}\label{prop:negative-obstruction}
For every $n\geq4$,
\[
 \min\{\sneg(G):G\text{ is a $2$-connected non-cycle graph of order }n\}
 =n-1.
\]
In particular, the unrestricted negative analogue of
Theorem~\ref{thm:classification} is false.
\end{proposition}

\begin{proof}
Theorem~\ref{thm:connected} gives $\sneg(G)\geq n-1$ for every
connected graph of order $n$.  On the other hand, $K_n$ is a
$2$-connected non-cycle graph for $n\geq4$.  Its eigenvalues are $n-1$ once
and $-1$ with multiplicity $n-1$.  Consequently,
$\sneg(K_n)=n-1$, proving both sharpness and the failure of an $n$
lower bound.
\end{proof}

\begin{remark}\label{rem:negative-obstruction}
Excluding complete graphs alone does not repair the assertion.  For
example,
\[
 \operatorname{Spec}(K_4-e)
 =
 \left\{\frac{1+\sqrt{17}}2,\frac{1-\sqrt{17}}2,-1,0\right\},
\]
and hence
\[
 \sneg(K_4-e)=\frac{11-\sqrt{17}}2<4.
\]
For $K_n$ one has $N=I-J/n$, and therefore
\[
 \mu_v(N)=1-\frac1{n^2}<2
\]
at every vertex.  Thus the triangle-density argument for $P$ cannot
be reflected to the negative spectral part.
\end{remark}

\begin{theorem}\label{thm:negative-classification}
Let $G$ be a triangle-free $2$-connected graph of order $n$.  Then
\[
 \sneg(G)\geq n
 \quad {\it if~and~only~if}\quad
 G\ncong C_{4k+3}\quad\text{for every }k\geq1.
\]
Moreover, equality holds if and only if $G$ is an even cycle, and
every non-cycle graph satisfies $\sneg(G)>n$.
\end{theorem}

\begin{proof}
Every $2$-connected graph has minimum degree at least 2.  If
$\Delta(G)=2$, then $G$ is a cycle, and the assertions follow from
Theorem~\ref{prop:cycles}.  If $\Delta(G)\geq3$, then
Theorem~\ref{thm:main} gives the strict inequality
$\sneg(G)>n$.  Since $C_3$ is not triangle-free, the exceptional
cycles in the present class are precisely $C_{4k+3}$ with $k\geq1$.
\end{proof}

\begin{corollary}\label{cor:hamiltonian}
Every Hamiltonian graph $G$ of order $n$ which is not a cycle satisfies
$
 \spen(G)\geq n.
$
If $G$ is also triangle-free, then
\[
 \min\{\spen(G),\sneg(G)\}>n.
\]
\end{corollary}

\section{Declaration of AI Usage}

The authors used an artificial intelligence tool during the preparation of this paper. The tool was employed to improve the English language and to assist with certain computational tasks. All results, derivations, and conclusions were independently verified by the authors. The authors accept full responsibility for the correctness of the manuscript.

\end{document}